\documentclass{article}
\usepackage[final]{graphicx}
\usepackage{psfrag}
\usepackage{amsmath,amsfonts,amssymb,amsxtra,subeqnarray}

\newcommand {\Kinf}{\ensuremath{\mathcal{K}_\infty}}

\newcommand {\Real}{\ensuremath{{\mathbb{R}}}}
\newcommand {\Natural}{\ensuremath{{\mathbb{N}}}}

\newcommand{\C}{\ensuremath{\mathcal C}}

\newcommand{\D}{\ensuremath{\mathcal D}}
\newcommand{\setO}{\ensuremath{\mathcal O}}

\newcommand{\U}{\ensuremath{\mathcal U}}

\newcommand{\X}{\ensuremath{\mathcal X}}
\newcommand{\Y}{\ensuremath{\mathcal Y}}

\newcommand{\ex}{\ensuremath{{\mathbf{x}}}}

\newcommand{\ef}{\ensuremath{{\mathbf{f}}}}

\newcommand{\tx}{\tilde x}
\newcommand{\xhat}{\hat x}

\newtheorem{theorem}{Theorem}

\newtheorem{remark}{Remark}
\newtheorem{assumption}{Assumption}

\newenvironment{proof}{\noindent {\bf Proof.}}{\hfill \hspace*{1pt}\hfill$\blacksquare$}

\begin{document}

\title{A dual pair of optimization-based formulations for estimation and control}
\author{S. Emre Tuna\footnote{The author is with Department of
Electrical and Electronics Engineering, Middle East Technical
University, 06800 Ankara, Turkey. Email: {\tt
tuna@eee.metu.edu.tr}}} \maketitle

\begin{abstract}

A finite-horizon optimal estimation problem for discrete-time
linear systems is formulated and solved. The formulation is a
natural extension of that which yields a deadbeat observer. The
resultant observer is the dual of the controller produced by the
finite-horizon minimum energy control problem with terminal
equality constraint. Nonlinear extensions of this dual pair are
also considered and sufficient conditions are provided for
stability and convergence.

\end{abstract}

\section{Introduction}

One of the earliest things that students of control theory are
taught is that for linear systems controllability and
observability are {\em dual} concepts. Very few doubt this because
it is in every linear systems textbook. Interestingly, what is
usually not in all those books is a clear definition of duality
\cite{luenberger92}. A possibility is that no one wants to confine
the notion into the precision required by a definition. Or,
perhaps, it is too obvious a thing to define. Either way, people
do not seem to need its exact description in order to make use of
or enjoy duality; for once a dual pair emerges, the human eye is
very quick to recognize it.

An intriguing example of duality is between the problems of linear
quadratic regulation (LQR) and linear quadratic estimation (LQE,
Kalman-Bucy filter). These celebrated optimization problems, which
are very different conceptually and formulation-wise, yield sets
of parameters (matrices) that are associated via formal rules that
transform one set to another \cite{kalman60}.\footnote[2]{Though
LQR and LQE are acknowledged as a dual pair, nowhere (to the best
of our knowledge) it is mentioned whether duality played much (if
any) role in their discoveries. In other words, there seems to be
no evidence to suggest that the birth of LQE was a consequence of
the pressing fact that LQR must have a {\em twin}.} The problems
of linear deadbeat control and linear deadbeat estimation make
another example of a dual pair. Let us recall the former. Consider
the below systems, both $n$th order,
\begin{eqnarray}
x_{k+1}&=&Ax_{k} \label{eqn:systemA} \\
\xhat_{k+1}&=&A\xhat_{k}+Bu_{k} \label{eqn:systemAB}
\end{eqnarray}
where the system~\eqref{eqn:systemAB} is to track the
system~\eqref{eqn:systemA} by choosing suitable control inputs
$u_{0},\,u_{1},\,\ldots$ (Let us assume for now that the
controllability condition is satisfied, input $u$ is scalar, and
the full state information $(\xhat,\,x)$ of both systems is
available to the controller.) To turn the
system~\eqref{eqn:systemAB} into a deadbeat tracker for the
system~\eqref{eqn:systemA}, i.e., to achieve $\xhat_{k}=x_{k}$ for
$k\geq n$, one can follow either of the below methods.
\begin{itemize}
\item[(M1)] Apply $u_{k}=K(x_{k}-\xhat_{k})$ where the row vector
$K$ is such that all the eigenvalues of $A-BK$ are at the origin.
\item[(M2)] Apply $u_{k}$ from the sequence of inputs
$(u_{k},\,u_{k+1},\,\ldots,\,u_{k+n-1})$ obtained by solving
$\xhat_{k+n}=A^{n}x_{k}$.
\end{itemize}
These methods are mathematically equivalent since, in the end,
they result in the same thing. However, the latter is superior to
the former in the following sense. Firstly, the feedback gain $K$
naturally comes out of the solution of
$A^{n}x_{k}=\xhat_{k+n}=A^{n}\xhat_{k}+A^{n-1}Bu_{k}+\ldots+ABu_{k+n-2}+Bu_{k+n-1}$.
Note that the first method does not give any clues as regards to
the computation of $K$. Secondly, and more importantly, the second
method is meaningful also for nonlinear systems, which is not the
case with the first one.

If we now move to the dual problem, linear deadbeat estimation,
the translation of the first statement~(M1) is well known. It
boils down to something like ``Choose an observer gain (say $L$)
such that all the eigenvalues of the matrix describing the error
dynamics (say $A-LC$) are at the origin.'' However, how to
translate the more valuable second statement~(M2) is not
immediately clear. Motivated by the historical pattern that
beautiful things tend to come in dual pairs for linear systems,
our work here starts with a search for this missing twin of (M2).
In more exact terms, guided by linear duality, we look for some
sort of a principle that not only leads to linear deadbeat
observer but also is useful for nonlinear deadbeat observer
design. This search is nothing but a simple linear algebra
exercise, but its outcome turns out to have some interesting
consequences that go beyond {\em linear} and {\em deadbeat}. Those
consequences are what we report in this paper. In particular,
three things are done:

First. In Theorem~\ref{thm:asyobslin} observer design for linear
systems is formulated as a finite-horizon optimization problem.
The formulation concerns a moving-horizon type observer (whose
order matches that of the system being observed) where at each
time an estimate of the system state is generated based solely on
the current output (instead of a larger collection of data
comprising previous measurements) of the system and the current
observer state. Convergence is guaranteed for all horizon lengths
no smaller than the order of the system being observed.
Interestingly, the formulation presented here turns out to be the
dual of a classic result (Theorem~\ref{thm:asyconlin}) by Kleinman
\cite{kleinman74}, who is acknowledged to be the first to consider
moving-horizon feedback \cite{keerthi88}.

Second. In Theorem~\ref{thm:asyobs} a nonlinear generalization of
the linear optimal observer construction of
Theorem~\ref{thm:asyobslin} is provided, where convergence is
established under certain conditions inspired by those that hold
in the linear problem. The resulting nonlinear moving-horizon
observer, like its above-mentioned linear version, is driven only
by the current output value of the system being observed. This
constitutes a conceptual difference between the construction in
this paper and the majority of the work on moving horizon
estimation \cite{rao03,alessandri08,psiaki13}, the basic
philosophy of which is summarized in \cite{ferrari02} as: {\em The
estimates of the states are obtained by solving a least squares
problem, which penalizes the deviation between measurements and
predicted outputs of a system. The data considered for the
optimization is laying in a window of fixed finite length, which
slides forward in time.}

Third. For the sake of symmetry we present in
Theorem~\ref{thm:asycon} a possible nonlinear extension of
Kleinman's optimal controller (Theorem~\ref{thm:asyconlin}). More
specifically, a moving-horizon optimal tracking problem is
considered, where convergence is established mainly through
terminal equality constraint. We note that more general results,
i.e., ones that do not require terminal equality constraint or
terminal cost, have long existed in the receding horizon control
literature \cite{mayne00,grimm05}.

\section{Notation}

$\Natural$ denotes the set of nonnegative integers and
$\Real_{\geq 0}$ the set of nonnegative real numbers. For a
mapping $f:\X\to\X$ let $f^{0}(x)=x$ and $f^{k+1}(x)=f(f^{k}(x))$.
Euclidean norm in $\Real^{n}$ is denoted by $\|\cdot\|$. For a
symmetric positive definite matrix $Q\in\Real^{n\times n}$ the
smallest and largest eigenvalues of $Q$ are respectively denoted
by $\lambda_{\rm min}(Q)$ and $\lambda_{\rm max}(Q)$. Also,
$\|x\|_{Q}^{2}=x^{T}Qx$. A function $\alpha:\Real_{\geq 0}\to
\Real_{\geq 0}$ is said to belong to class-$\Kinf$
$(\alpha\in\Kinf)$ if it is continuous, zero at zero, strictly
increasing, and unbounded.

\section{An optimal observer}

We begin this section by an attempt to obtain the dual of the
statement (M2), i.e., some method to construct deadbeat observer,
which is meaningful also for nonlinear systems. Consider the
discrete-time linear system
\begin{eqnarray}\label{eqn:systemCA}
x^{+}=Ax\,,\quad y=Cx
\end{eqnarray}
where $x\in\Real^{n}$ is the {\em state}, $y\in\Real^{m}$ is the
{\em output}, and $x^{+}$ is the state at the next time instant.
The matrices $A$ and $C$ belong to $\Real^{n\times n}$ and
$\Real^{m\times n}$, respectively. We will denote the solution of
the system~\eqref{eqn:systemCA} by $x_{k}$ for $k\in\Natural$.
Driven by the output $y$ of the system~\eqref{eqn:systemCA}
suppose that the below system, for $N\geq 1$,
\begin{eqnarray}\label{eqn:preobserver}
z^{+}=A\eta(z,\,y)
\end{eqnarray}
produces at each time $k$ an estimate $z_{k}$ of $x_{k-N+1}$ (the
$N-1$ steps earlier value of the current state $x$) based on
$z_{k-1}$ and $y_{k-1}$. That is, the vector $\eta\in\Real^{n}$ is
a function of the state $z$ and the output $y$. Note that the
system~\eqref{eqn:preobserver} can be used in the following
observer
\begin{eqnarray}\label{eqn:observer}
z^{+}=A\eta\,,\quad \xhat=A^{N-1}z
\end{eqnarray}
where $\xhat$ is the estimate of the current state $x$. Assuming
for now that the system~\eqref{eqn:systemCA} is observable and its
output $y$ is scalar, we now ask the following question. How
should $\eta$ be chosen such that the system~\eqref{eqn:observer}
is a deadbeat observer for the system~\eqref{eqn:systemCA}, i.e.,
$\xhat_{k}=x_{k}$ for $k\geq n$?

To answer the question we recall the deadbeat tracker, the dual of
deadbeat observer. From (M2) it follows that the dynamics of the
deadbeat tracker read
\begin{eqnarray*}
\xhat^{+}=A\xhat+BK(x-\xhat)
\end{eqnarray*}
with the feedback gain
\begin{eqnarray*}
K=e_{n}^{T}\C^{-1}A^{n}\,.
\end{eqnarray*}
where $\C=[B\ AB\ \ldots\ A^{n-1}B]$ is the controllability matrix
and $e_{n}=[0\ \ldots\ 0\ 1]^{T}$. By duality the dynamics of the
deadbeat observer should read
\begin{eqnarray}\label{eqn:robin}
\xhat^{+}=A\xhat+L(y-C\xhat)
\end{eqnarray}
with the observer gain
\begin{eqnarray}\label{eqn:thrush}
L=A^{n}\setO^{-1}e_{n}
\end{eqnarray}
where $\setO=[C^{T}\ A^{T}C^{T}\ \ldots\ A^{(n-1)T}C^{T}]^{T}$ is
the observability matrix. Now, combining \eqref{eqn:observer},
\eqref{eqn:robin}, and \eqref{eqn:thrush} we can write
\begin{eqnarray*}
A^{N}\eta
&=&A^{N-1}z^{+}\\
&=&\xhat^{+}\\
&=&A\xhat+A^{n}\setO^{-1}e_{n}(y-C\xhat)\\
&=&A^{N}z+A^{n}\setO^{-1}e_{n}(y-CA^{N-1}z)\,.
\end{eqnarray*}
If we let $N=n$ we can write
\begin{eqnarray*}
A^{n}\eta =A^{n}(z+\setO^{-1}e_{n}(y-CA^{n-1}z))
\end{eqnarray*}
which suggests we choose $\eta$ as
\begin{eqnarray}\label{eqn:silver}
\eta=z+\setO^{-1}e_{n}(y-CA^{n-1}z)\,.
\end{eqnarray}
Equation~\eqref{eqn:silver} is not directly generalizable to
nonlinear systems so we rewrite it as the following set of
equations
\begin{eqnarray}\label{eqn:golden}
\left.
\begin{array}{rcl}
C\eta&=&Cz\\
CA\eta&=&CAz\\
&\vdots&\\
CA^{n-2}\eta&=&CA^{n-2}z\\
CA^{n-1}\eta&=&y
\end{array}\right\}
\end{eqnarray}
Therefore, to turn the system~\eqref{eqn:observer} (for $N=n$)
into a deadbeat observer for the system~\eqref{eqn:systemCA} one
can use the below algorithm.
\begin{itemize}
\item[(M3)] Choose $\eta_{k}$ such that the would-be future output
values $CA^{i}\eta_{k}$ match the would-be future output values of
the current observer state $CA^{i}z_{k}$ for
$i=0,\,1,\,\ldots,\,n-2$; and the would-be future output value
$CA^{n-1}\eta_{k}$ matches the current measurement $y_{k}$.
\end{itemize}

The statement (M3) seems to be the dual of (M2). Happily, it
serves our purpose in the sense that it allows one to construct
nonlinear deadbeat observers. The formal treatment of the case is
as follows.

Consider the system
\begin{eqnarray}\label{eqn:nonlinear}
x^{+}=f(x)\,,\quad y=h(x)
\end{eqnarray}
with $f:\X\to\X$ and $h:\X\to\Y$. Now consider the observer system
\begin{eqnarray}\label{eqn:nlobserver}
z^{+}=f(\eta)\,,\quad \xhat=f^{N-1}(z)
\end{eqnarray}
for some integer $N\geq 1$.

\begin{assumption}\label{assume:dbobs}
For each $\xi\in\Y^{N}$ the equation
\begin{eqnarray*}
\left[\begin{array}{c}h(\eta)\\h(f(\eta))\\\vdots\\h(f^{N-1}(\eta))\end{array}\right]=\xi
\end{eqnarray*}
has a unique solution $\eta\in\X$.
\end{assumption}
Note that when the system~\eqref{eqn:nonlinear} is linear with
scalar output, Assumption~\ref{assume:dbobs} (with $N=n$) is
equivalent to observability. The linear statement (M3) leads to
the following result by Glad \cite{glad83}. For a geometric
interpretation see \cite{tuna12}.

\begin{theorem}\label{thm:db}
Consider the system~\eqref{eqn:nonlinear} and the
observer~\eqref{eqn:nlobserver}. Suppose
Assumption~\ref{assume:dbobs} holds and let $\eta$ be chosen to
satisfy
\begin{eqnarray*}
h(\eta)&=&h(z)\\
h(f(\eta))&=&h(f(z))\\
&\vdots&\\
h(f^{N-2}(\eta))&=&h(f^{N-2}(z))\\
h(f^{N-1}(\eta))&=&y\,.
\end{eqnarray*}
Then, for all initial conditions, $\xhat_{k}=x_{k}$ for all $k\geq
N$.
\end{theorem}

\begin{proof}
The result follows trivially for $N=1$. Suppose now $N\geq 2$ and
for some $p\in\{1,\,2,\,\ldots,\,N-1\}$ and some $k\geq 0$ we have
\begin{eqnarray}\label{eqn:seed}
h(f^{N-q}(\eta_{k}))=y_{k-q+1}\quad\forall
q\in\{1,\,2,\,\ldots,\,p\}\,.
\end{eqnarray}
Then we can write
\begin{eqnarray*}
h(f^{N-q-1}(\eta_{k+1}))
&=& h(f^{N-q-1}(z_{k+1}))\nonumber\\
&=& h(f^{N-q-1}(f(\eta_{k})))\nonumber\\
&=& h(f^{N-q}(\eta_{k}))\nonumber\\
&=& y_{k-q+1}\,.
\end{eqnarray*}
Also, $h(f^{N-1}(\eta_{k+1}))=y_{k+1}$ holds by definition. Hence
\eqref{eqn:seed} implies
\begin{eqnarray*}
h(f^{N-q}(\eta_{k+1}))=y_{(k+1)-q+1}\quad\forall
q\in\{1,\,2,\,\ldots,\,p+1\}\,.
\end{eqnarray*}
Now, \eqref{eqn:seed} holds at time $k=0$ for $p=1$. By induction
therefore we can write
\begin{eqnarray*}
\left[\begin{array}{c}
h(\eta_{k})\\
\vdots\\
h(f^{N-2}(\eta_{k}))\\
h(f^{N-1}(\eta_{k}))
\end{array}\right]
= \left[\begin{array}{c}
y_{k-N+1}\\
\vdots\\
y_{k-1}\\
y_{k}
\end{array}\right]
= \left[\begin{array}{c}
h(x_{k-N+1})\\
\vdots\\
h(f^{N-2}(x_{k-N+1}))\\
h(f^{N-1}(x_{k-N+1}))
\end{array}\right]
\end{eqnarray*}
for all $k\geq N-1$. Then by Assumption~\ref{assume:dbobs} we have
$\eta_{k}=x_{k-N+1}$ for all $k\geq N-1$. The result follows since
$\xhat_{k+1}=f^{N}(\eta_{k})$.
\end{proof}
\medskip\smallskip

As Theorem~\ref{thm:db} depicted, the rationale behind the set of
linear equations~\eqref{eqn:golden} allows us to construct a
nonlinear deadbeat observer. What else can we get out of
\eqref{eqn:golden}? Now we attempt to answer this question.

Let us once again consider the system~\eqref{eqn:systemCA}
together with the observer~\eqref{eqn:observer} and write the
general version of \eqref{eqn:golden}
\begin{eqnarray}\label{eqn:golden2}
\left.
\begin{array}{rcl}
C\eta&=&Cz\\
CA\eta&=&CAz\\
&\vdots&\\
CA^{N-2}\eta&=&CA^{N-2}z\\
CA^{N-1}\eta&=&y
\end{array}\right\}
\end{eqnarray}
where $N$ need not equal the order of the
system~\eqref{eqn:systemCA}. Suppose now the set of
equations~\eqref{eqn:golden2} is overdetermined and does not admit
a solution $\eta$. How to choose $\eta$ then? Any textbook on
linear algebra would suggest the least squares approximation,
which leads to the following result.

\begin{theorem}\label{thm:asyobslin}
Consider the system~\eqref{eqn:systemCA}. Let $N\geq 1$ be such
that the matrix $[C^{T}\ A^{T}C^{T}\ \ldots\ A^{(N-1)T}C^{T}]$ is
full row rank. Let $R\in\Real^{m\times m}$ be a symmetric positive
definite matrix and consider the observer~\eqref{eqn:observer}
with $\eta=\arg\min_{\xi}J(\xi,\,z,\,y)$ where
\begin{eqnarray}\label{eqn:golden2optim}
J(\xi,\,z,\,y):=
\|CA^{N-1}\xi-y\|_{R}^{2}+\sum_{i=0}^{N-2}\|CA^{i}\xi-CA^{i}z\|_{R}^{2}\,.
\end{eqnarray}
Then $\|\xhat_{k}-x_{k}\|\to 0$ as $k\to\infty$.
\end{theorem}

\begin{proof}
Let us define the symmetric matrices
\begin{eqnarray*}
Q&:=&C^{T}RC+A^{T}C^{T}RCA+\ldots+A^{(N-2)T}C^{T}RCA^{N-2}\\
H&:=&A^{(N-1)T}C^{T}RCA^{N-1}\,.
\end{eqnarray*}
Note that by rank assumption the matrix $Q+H$ is nonsingular.
Solving $\partial J/\partial\xi=0$ we obtain
\begin{eqnarray}\label{eqn:eta}
\eta=(Q+H)^{-1}A^{(N-1)T}C^{T}Ry+(Q+H)^{-1}Qz\,.
\end{eqnarray}
Let us define the shorthand notation ${\tilde x}_{k}:=x_{k-N+1}$
for $k\geq N-1$. Then we have $x=A^{N-1}{\tilde x}$ and
$y=CA^{N-1}{\tilde x}$. We can now rewrite \eqref{eqn:eta} as
\begin{eqnarray}\label{eqn:forL}
\eta&=&(Q+H)^{-1}H{\tilde x}+(Q+H)^{-1}Qz\nonumber\\
&=&z+(Q+H)^{-1}H(\tx-z)\\
&=&\tx+(Q+H)^{-1}Q(z-\tx)\,.\nonumber
\end{eqnarray}
Then
\begin{eqnarray}
\eta-{\tilde x}&=&(Q+H)^{-1}Q(z-{\tilde x})\label{eqn:etaminusxtilde}\\
\eta-z&=&(Q+H)^{-1}H({\tilde x}-z)\nonumber
\end{eqnarray}
and we can write
\begin{eqnarray}\label{eqn:laurel}
J(\eta,\,z,\,y)&=&(\eta-{\tilde x})^{T}H(\eta-{\tilde
x})+(\eta-z)^{T}Q(\eta-z)\nonumber\\
&=&(z-{\tilde x})^{T}Q(Q+H)^{-1}H(Q+H)^{-1}Q(z-{\tilde
x})\nonumber\\
&&+(z-{\tilde x})^{T}H(Q+H)^{-1}Q(Q+H)^{-1}H(z-{\tilde x})\,.
\end{eqnarray}
Note that
\begin{eqnarray}\label{eqn:hardy}
Q(Q+H)^{-1}H
&=&(Q+H-H)(Q+H)^{-1}(Q+H-Q) \nonumber\\
&=&H(Q+H)^{-1}Q\,.
\end{eqnarray}
Combining \eqref{eqn:laurel} and \eqref{eqn:hardy} we can write
\begin{eqnarray}\label{eqn:atos}
J(\eta,\,z,\,y)&=&(z-{\tilde
x})^{T}Q(Q+H)^{-1}H(Q+H)^{-1}Q(z-{\tilde
x})\nonumber\\
&&+(z-{\tilde x})^{T}H(Q+H)^{-1}H(Q+H)^{-1}Q(z-{\tilde x})\nonumber\\
&=&(z-{\tilde
x})^{T}(Q(Q+H)^{-1}+H(Q+H)^{-1})H(Q+H)^{-1}Q(z-{\tilde x})\nonumber\\
&=&(z-{\tilde
x})^{T}(Q+H)(Q+H)^{-1}H(Q+H)^{-1}Q(z-{\tilde x})\nonumber\\
&=&(z-{\tilde x})^{T}H(Q+H)^{-1}Q(z-{\tilde x})\,.
\end{eqnarray}
Now, by \eqref{eqn:etaminusxtilde} we can write
\begin{eqnarray}\label{eqn:portos}
(\eta-\tx)^{T}(Q+H)(\eta-\tx)&=&(z-\tx)^{T}Q(Q+H)^{-1}Q(z-\tx)\,.
\end{eqnarray}
Then, by \eqref{eqn:atos} and \eqref{eqn:portos} we have
\begin{eqnarray}\label{eqn:aramis}
J(\eta,\,z,\,y)+(\eta-\tx)^{T}(Q+H)(\eta-\tx)=(z-\tx)^{T}Q(z-\tx)\,.
\end{eqnarray}
Since
\begin{eqnarray*}
A^{T}QA&=&Q+H-C^{T}RC\\
&\leq&Q+H
\end{eqnarray*}
we can write by \eqref{eqn:aramis}
\begin{eqnarray}\label{eqn:lyap}
(z^{+}-\tx^{+})^{T}Q(z^{+}-\tx^{+})
&=&(\eta-\tx)^{T}A^{T}QA(\eta-\tx) \nonumber\\
&\leq&(\eta-\tx)^{T}(Q+H)(\eta-\tx)\nonumber\\
&=&(z-\tx)^{T}Q(z-\tx)-J(\eta,\,z,\,y)\,.
\end{eqnarray}
Note that \eqref{eqn:lyap} could serve as a Lyapunov inequality if
$Q$ were positive definite, which we do not assume. Still,
\eqref{eqn:lyap} is whence we extract stability. First we need to
demonstrate the following.

{\em Claim:} For each $\varepsilon\geq 0$ there exists $\delta\geq
0$ such that for all $k_{1}\in\Natural$
\begin{eqnarray}\label{eqn:claim}
J(\eta_{k},\,z_{k},\,y_{k})\leq\delta\quad \forall k\geq
k_{1}\implies\|\xhat_{k}-x_{k}\|\leq\varepsilon\quad \forall k\geq
k_{1}+N\,.
\end{eqnarray}
We prove this claim as follows. Let us for some $\delta$ and
$k_{1}$ have $J(\eta_{k},\,z_{k},\,y_{k})\leq\delta$ for all
$k\geq k_{1}$, which by \eqref{eqn:golden2optim} implies
\begin{eqnarray*}
\max\
\{\|CA^{N-1}\eta_{k}-y_{k}\|,\,\|CA^{N-2}(\eta_{k}-z_{k})\|,\,\ldots,\,\|C(\eta_{k}-z_{k})\|\}\leq\delta_1
\end{eqnarray*}
with $\delta_{1}=\sqrt{\delta/\lambda_{\rm min}(R)}$. The claim is
evident for $N=1$. Consider now $N\geq 2$ and suppose for some
$p\in\{1,\,2,\,\ldots,\,N-1\}$ and some $k_{p}\geq k_{1}$ we have
\begin{eqnarray}\label{eqn:claimproof1}
\|CA^{N-q}\eta_{k}-y_{k-q+1}\|\leq q\delta_{1} \quad\forall
q\in\{1,\,2,\,\ldots,\,p\}\quad \forall k\geq k_{p}\,.
\end{eqnarray}
Then we can write for $q\in\{1,\,2,\,\ldots,\,p\}$ and $k\geq
k_{p}$
\begin{eqnarray*}
\|CA^{N-(q+1)}\eta_{k+1}-y_{(k+1)-(q+1)+1}\|&\leq&
\|CA^{N-q-1}(\eta_{k+1}-z_{k+1})\|\\
&&+\|CA^{N-q-1}z_{k+1}-y_{k-q+1}\|\\
&=&
\|CA^{N-q-1}(\eta_{k+1}-z_{k+1})\|\\&&+\|CA^{N-q}\eta_{k}-y_{k-q+1}\|\\
&\leq&\delta_{1}+q\delta_{1}\\
&=&(q+1)\delta_{1}
\end{eqnarray*}
which allows us to assert
\begin{eqnarray*}
\|CA^{N-q}\eta_{k}-y_{k-q+1}\|\leq q\delta_{1} \quad\forall
q\in\{1,\,2,\,\ldots,\,p+1\}\quad\forall k\geq k_{p}+1\,.
\end{eqnarray*}
Since \eqref{eqn:claimproof1} holds with $p=1$, by induction we
can write
\begin{eqnarray}\label{eqn:claimproof2}
\|CA^{N-q}\eta_{k}-y_{k-q+1}\|\leq q\delta_{1} \quad\forall
q\in\{1,\,2,\,\ldots,\,N\}\quad\forall k\geq k_{1}+N-1\,.
\end{eqnarray}
Define the matrix $W:=[C^{T}\ A^{T}C^{T}\ \ldots\
A^{(N-1)T}C^{T}]^{T}$. Now by \eqref{eqn:claimproof2} we can write
\begin{eqnarray*}
\|\xhat_{k+1}-x_{k+1}\|^{2}
&=&\|A^{N}(\eta_{k}-\tx_{k})\|^{2}\\
&\leq&\lambda_{\max}(A^{NT}A^{N})\|(\eta_{k}-\tx_{k})\|^{2}\\
&\leq&\lambda_{\max}(A^{NT}A^{N})\lambda_{\min}^{-1}(W^{T}W)\|W(\eta_{k}-\tx_{k})\|^{2}\\
&=&\lambda_{\max}(A^{NT}A^{N})\lambda_{\min}^{-1}(W^{T}W)\sum_{q=1}^{N}\|CA^{N-q}\eta_{k}-y_{k-q+1}\|^{2}\\
&\leq&\lambda_{\max}(A^{NT}A^{N})\lambda_{\min}^{-1}(W^{T}W)\delta_{1}^{2}\sum_{q=1}^{N}q^{2}\\
&=&\frac{(2N^{3}+3N^{2}+N)\lambda_{\max}(A^{NT}A^{N})\delta}{6\lambda_{\min}(W^{T}W)\lambda_{\min}(R)}\,.
\end{eqnarray*}
This proves our claim because given any $\varepsilon$, we can
choose
\begin{eqnarray*}
\delta\leq\frac{6\lambda_{\min}(W^{T}W)\lambda_{\min}(R)\varepsilon^{2}}{(2N^{3}+3N^{2}+N)
\lambda_{\max}(A^{NT}A^{N})}
\end{eqnarray*}
to satisfy \eqref{eqn:claim}. Now we return to the proof of the
theorem. Observe that the inequality \eqref{eqn:lyap} implies that
the sum $\sum_{k=0}^{\infty}J(\eta_{k},\,z_{k},\,y_{k})$ is
bounded. Since the terms being summed are all nonnegative we must
have $J(\eta_{k},\,z_{k},\,y_{k})\to 0$ as $k\to\infty$, which by
\eqref{eqn:claim} yields $\|\xhat_{k}-x_{k}\|\to 0$ as
$k\to\infty$.
\end{proof}
\medskip\smallskip

We note that the optimal observer coming out of the formulation
depicted in Theorem~\ref{thm:asyobslin} enjoys the classic linear
observer structure $\xhat^{+}=A\xhat+L(y-C\xhat)$ with the
observer gain
\begin{eqnarray}\label{eqn:optimalL}
L=A^{N}(C^{T}RC+\ldots+A^{(N-1)T}C^{T}RCA^{N-1})^{-1}A^{(N-1)T}C^{T}R
\end{eqnarray}
following from \eqref{eqn:observer} and \eqref{eqn:forL}.
Theorem~\ref{thm:asyobslin} then implies that the eigenvalues of
the matrix $A-LC$ (the system matrix of the error dynamics) must
all be within the open unit disc.

\section{An optimal tracker}

The previous section started with a search for the principle behind
deadbeat observer. Our search was driven by the question {\em what
method would lead to the observer gain given in \eqref{eqn:thrush}},
where the gain \eqref{eqn:thrush} was obtained by duality from the
feedback gain of the deadbeat tracker. In this section we will employ
duality once again, this time however in the other direction. In
particular, we ask the following question. What is the dual of the
optimal observer described in Theorem~\ref{thm:asyobslin}? Or, more
directly, what is the optimization problem that leads to the following
feedback gain?
\begin{eqnarray}\label{eqn:optimalK}
K=RB^{T}A^{(N-1)T}(BRB^{T}+\ldots+A^{N-1}BRB^{T}A^{(N-1)T})^{-1}A^{N}
\end{eqnarray}
which we obtain from \eqref{eqn:optimalL} by duality. The answer
is the below result by Kleinman \cite{kleinman74}, which is
sometimes called the minimum energy control problem.

\begin{theorem}\label{thm:asyconlin}
Consider the system~\eqref{eqn:systemA} and the
tracker~\eqref{eqn:systemAB} with $A\in\Real^{n\times n}$ and
$B\in\Real^{n\times m}$. Let $N\geq 1$ be such that the matrix
$[B\ AB\ \ldots\ A^{N-1}B]$ is full row rank and let
$R\in\Real^{m\times m}$ be a symmetric positive definite matrix.
Let the control input of the tracker be $u=v_{0}(\xhat,\,x)$ where
$v_{0}$ is the first term of the sequence
$(v_{0},\,v_{1},\,\ldots,\,v_{N-1})$ satisfying
\begin{eqnarray*}
(v_{i})_{i=0}^{N-1}=\arg\min_{(w_{i})_{i=0}^{N-1}}\sum_{i=0}^{N-1}\|w_{i}\|_{R^{-1}}^{2}\quad\mbox{subject
to}\quad
\left\{\begin{array}{rcl}z_{0}&=&\xhat\\z_{i+1}&=&Az_{i}+Bw_{i}\\z_{N}&=&A^{N}x\end{array}\right.
\end{eqnarray*}
Then $\|\xhat_{k}-x_{k}\|\to 0$ as $k\to\infty$.
\end{theorem}

\begin{proof}
Given $\xhat$ and $x$, one can show that $v_{0}=K(x-\xhat)$ with
$K$ given in \eqref{eqn:optimalK}. In the light of duality
convergence then follows from Theorem~\ref{thm:asyobslin} and
\eqref{eqn:optimalL}.
\end{proof}

\section{Nonlinear formulations}

In this section we present possible nonlinear extensions of the
linear formulations described earlier in the paper. First, for the
observer design problem we will propose an optimization-based
formulation that leads to desired observer behavior under certain
sufficient conditions. Then we will repeat the procedure for the
tracker design problem. Throughout this section the pairs
$(\X,\,\rho_{\rm x})$ and $(\Y,\,\rho_{\rm y})$ will denote
finite-dimensional complete metric spaces \cite{browder96}.
\medskip\smallskip

\noindent{\bf Caveat.} Henceforth we will avoid the standard use
of parentheses when the risk of confusion is negligible. For
instance, $h(f(x))$ will be replaced by $hfx$.

\subsection{Observer design}

Consider the system~\eqref{eqn:nonlinear} and the
observer~\eqref{eqn:nlobserver}. We let $f$ and $h$ be uniformly
continuous functions. Let $\ell:\Y\times\Y\to\Real_{\geq 0}$ and
$\alpha_{1},\,\alpha_{2}\in\Kinf$ satisfy $\alpha_{1}\rho_{\rm
y}(v,\,w)\leq\ell(v,\,w)\leq\alpha_{2}\rho_{\rm y}(v,\,w)$ for
every $v,\,w\in\Y$. There is no harm in assuming the symmetry
$\ell(v,\,w)=\ell(w,\,v)$. Now we define the cost function
$J:\X\times\X\times\Y\to\Real_{\geq 0}$ as
\begin{eqnarray}\label{eqn:cost}
J(\xi,\,z,\,y):=\ell(hf^{N-1}\xi,\,y)+\sum_{i=0}^{N-2}\ell(hf^{i}\xi,\,hf^{i}z)\,.
\end{eqnarray}

\begin{assumption}\label{assume:two}
The following hold.
\begin{enumerate}
\item There exists $\alpha_{3}\in\Kinf$ such that for all
$z,\,\tx\in\X$ we have
\begin{eqnarray}\label{eqn:uniformobs}
\sum_{i=0}^{N-1}\ell(hf^{i}z,\,hf^{i}\tx)\geq\alpha_{3}\rho_{\rm
x}(z,\,\tx)\,.
\end{eqnarray}
\item There exists $\eta:\X\times\Y\to\X$ such that
$J(\eta(z,\,y),\,z,\,y)<J(\xi,\,z,\,y)$ for all
$\xi\neq\eta(z,\,y)$. Moreover, there exists $\alpha_{4}\in\Kinf$
such that for all $z,\,\tx\in\X$ we have
\begin{eqnarray}\label{eqn:Jdecay}
\alpha_{4}J(\eta,\,z,\,hf^{N-1}\tx)+
\sum_{i=0}^{N-1}\ell(hf^{i}\eta,\,hf^{i}\tx)
\leq\sum_{i=0}^{N-2}\ell(hf^{i}z,\,hf^{i}\tx)
\end{eqnarray}
where $\eta=\eta(z,\,hf^{N-1}\tx)$.
\end{enumerate}
\end{assumption}

\begin{remark}
The linear case studied in Theorem~\ref{thm:asyobslin} inspires
the conditions listed in Assumption~\ref{assume:two}.  In
particular, the first condition is a characterization of (uniform)
observability and the second condition attempts to translate
\eqref{eqn:aramis} to the nonlinear setting.
\end{remark}

\begin{theorem}\label{thm:asyobs}
Consider the system~\eqref{eqn:nonlinear}. Let $N\geq 1$ and
Assumption~\ref{assume:two} hold. Consider the
observer~\eqref{eqn:nlobserver} with
$\eta=\arg\min_{\xi}J(\xi,\,z,\,y)$. Then for each $\varepsilon>0$
there exists $\delta>0$ such that $\rho_{\rm
x}(\xhat_{0},\,x_{0})\leq\delta$ implies $\rho_{\rm
x}(\xhat_{k},\,x_{k})\leq\varepsilon$ for all $k\in\Natural$.
Moreover, for all initial conditions, $\rho_{\rm
x}(\xhat_{k},\,x_{k})\to 0$ as $k\to\infty$.
\end{theorem}

\begin{proof}
Note that if $N=1$ then $\eta=\arg\min_{\xi}\ell(h\xi,\,hx)$ and
by uniqueness assumption we must have $\eta=x$. Therefore
$\xhat_{k}=x_{k}$ for all $k\geq 1$ and the result follows
trivially. In the sequel we suppose $N\geq 2$.

We begin by stability. Since $f$ and $h$ are uniformly continuous
there exist $\alpha_{5},\,\alpha_{6}\in\Kinf$ such that $\rho_{\rm
x}(f\xi,\,f\zeta)\leq\alpha_{5}\rho_{\rm x}(\xi,\,\zeta)$ and
$\rho_{\rm y}(h\xi,\,h\zeta)\leq\alpha_{6}\rho_{\rm
x}(\xi,\,\zeta)$ for all $\xi,\,\zeta\in\X$. Then we have
\begin{eqnarray}
\rho_{\rm x}(\xhat^{+},\,x^{+})
&=&\rho_{\rm x}(f^{N}\eta,\,fx)\nonumber\\
&\leq&\alpha_{5}\rho_{\rm x}(f^{N-1}\eta,\,x)\nonumber\\
&\leq&\alpha_{5}(\rho_{\rm x}(f^{N-1}\eta,\,f^{N-1}z)+\rho_{\rm x}(f^{N-1}z,\,x))\nonumber\\
&=&\alpha_{5}(\rho_{\rm x}(f^{N-1}\eta,\,f^{N-1}z)+\rho_{\rm x}(\xhat,\,x))\nonumber\\
&=&\alpha_{5}(\alpha_{5}^{N-1}\rho_{\rm x}(\eta,\,z)+\rho_{\rm
x}(\xhat,\,x))\,.\label{eqn:boundedness1}
\end{eqnarray}
By \eqref{eqn:uniformobs} and the fact that $J(\eta,\,z,\,y)\leq
J(z,\,z,\,y)=\ell(h\xhat,\,hx)$ we can proceed as
\begin{eqnarray}
\rho_{\rm
x}(\eta,\,z)&\leq&\alpha_{3}^{-1}\sum_{i=0}^{N-1}\ell(hf^{i}\eta,\,hf^{i}z)\nonumber\\
&=&\alpha_{3}^{-1}(J(\eta,\,z,\,y)+\ell(hf^{N-1}\eta,\,hf^{N-1}z)-\ell(hf^{N-1}\eta,\,hx))\nonumber\\
&\leq&\alpha_{3}^{-1}(\ell(h\xhat,\,hx)+\ell(hf^{N-1}\eta,\,hf^{N-1}z))\nonumber\\
&\leq&\alpha_{3}^{-1}(\alpha_{2}\alpha_{6}\rho_{\rm
x}(\xhat,\,x)+\ell(hf^{N-1}\eta,\,hf^{N-1}z))\,.\label{eqn:boundedness2}
\end{eqnarray}
Moreover,
\begin{eqnarray}
\ell(hf^{N-1}\eta,\,hf^{N-1}z)
&\leq&\alpha_{2}\rho_{\rm y}(hf^{N-1}\eta,\,hf^{N-1}z)\nonumber\\
&\leq&\alpha_{2}(\rho_{\rm y}(hf^{N-1}\eta,\,hx)+\rho_{\rm y}(hx,\,hf^{N-1}z))\nonumber\\
&\leq&\alpha_{2}(\alpha_{1}^{-1}\ell(hf^{N-1}\eta,\,hx)+\rho_{\rm y}(hx,\,h\xhat))\nonumber\\
&\leq&\alpha_{2}(\alpha_{1}^{-1}J(\eta,\,z,\,y)+\alpha_{6}\rho_{\rm x}(\xhat,\,x))\nonumber\\
&\leq&\alpha_{2}(\alpha_{1}^{-1}\ell(h\xhat,\,hx)+\alpha_{6}\rho_{\rm x}(\xhat,\,x))\nonumber\\
&\leq&\alpha_{2}(\alpha_{1}^{-1}\alpha_{2}\alpha_{6}\rho_{\rm
x}(\xhat,\,x)+\alpha_{6}\rho_{\rm
x}(\xhat,\,x))\,.\label{eqn:boundedness3}
\end{eqnarray}
Now let us define $\alpha_{7}\in\Kinf$ as
\begin{eqnarray*}
\alpha_{7}s:=\alpha_{5}(\alpha_{5}^{N-1}\alpha_{3}^{-1}(\alpha_{2}\alpha_{6}s+\alpha_{2}(\alpha_{1}^{-1}\alpha_{2}\alpha_{6}s+\alpha_{6}s))+s)\,.
\end{eqnarray*}
Then by \eqref{eqn:boundedness1}, \eqref{eqn:boundedness2}, and
\eqref{eqn:boundedness3} we can write
\begin{eqnarray}\label{eqn:boundedness4}
\rho_{\rm x}(\xhat^{+},\,x^{+})\leq\alpha_{7}\rho_{\rm
x}(\xhat,\,x)\,,
\end{eqnarray}
which tells us that if $x$ and its estimate $\xhat$ are close to
each other at some instant then they will stay close at the next
instant. Now we direct our attention to \eqref{eqn:Jdecay}. Before
however let us let $\tx$ indicate the $N-1$ time steps earlier
value of the state $x$, i.e., $\tx_{k}:=x_{k-N+1}$ for $k\geq
N-1$. Then we can write
\begin{eqnarray*}
\alpha_{4}J(\eta,\,z,\,hf^{N-1}\tx)
&\leq&\sum_{i=0}^{N-2}\ell(hf^{i}z,\,hf^{i}\tx)-\sum_{i=0}^{N-1}\ell(hf^{i}\eta,\,hf^{i}\tx)\\
&\leq&\sum_{i=0}^{N-2}\ell(hf^{i}z,\,hf^{i}\tx)-\sum_{i=1}^{N-1}\ell(hf^{i}\eta,\,hf^{i}\tx)\\
&=&\sum_{i=0}^{N-2}\ell(hf^{i}z,\,hf^{i}\tx)-\sum_{i=0}^{N-2}\ell(hf^{i}z^{+},\,hf^{i}\tx^{+})\,,
\end{eqnarray*}
which implies
\begin{eqnarray}\label{eqn:Jsumbounded}
\sum_{k=k_{0}}^{\infty}\alpha_{4}J(\eta_{k},\,z_{k},\,y_{k})\leq\sum_{i=0}^{N-2}\ell(hf^{i}z_{k_{0}},\,hf^{i}\tx_{k_{0}})
\end{eqnarray}
for all $k_{0}\geq N-1$. We now demonstrate the following.

{\em Claim:} There exists $\alpha_{8}\in\Kinf$ such that for all
$k_{1}\in\Natural$
\begin{eqnarray}\label{eqn:claimNL}
J(\eta_{k},\,z_{k},\,y_{k})\leq\delta\quad \forall k\geq
k_{1}\implies \rho_{\rm
x}(\xhat_{k},\,x_{k})\leq\alpha_{8}\delta\quad \forall k\geq
k_{1}+N\,.
\end{eqnarray}
We prove this claim as follows. Let us for some $\delta$ and
$k_{1}$ have $J(\eta_{k},\,z_{k},\,y_{k})\leq\delta$ for all
$k\geq k_{1}$, which by \eqref{eqn:cost} implies
\begin{eqnarray*}
\max\ \{\rho_{\rm y}(hf^{N-1}\eta_{k},\,y_{k}),\,\rho_{\rm
y}(hf^{N-2}\eta_{k},\,hf^{N-2}z_{k}),\,\ldots,\,\rho_{\rm
y}(h\eta_{k},\,hz_{k})\}\leq\delta_1
\end{eqnarray*}
with $\delta_{1}=\alpha_{1}^{-1}\delta$. Suppose for some
$p\in\{1,\,2,\,\ldots,\,N-1\}$ and some $k_{p}\geq k_{1}$ we have
\begin{eqnarray}\label{eqn:claimNLproof1}
\rho_{\rm y}(hf^{N-q}\eta_{k},\,y_{k-q+1})\leq q\delta_{1}
\quad\forall q\in\{1,\,2,\,\ldots,\,p\}\quad\forall k\geq k_{p}\,.
\end{eqnarray}
Then we can write for all $q\in\{1,\,2,\,\ldots,\,p\}$ and $k\geq
k_{p}$
\begin{eqnarray*}
\rho_{\rm y}(hf^{N-(q+1)}\eta_{k+1},\,y_{(k+1)-(q+1)+1})&\leq&
\rho_{\rm y}(hf^{N-q-1}\eta_{k+1},\,hf^{N-q-1}z_{k+1})\\
&&+\rho_{\rm y}(hf^{N-q-1}z_{k+1},\,y_{k-q+1})\\
&=&
\rho_{\rm y}(hf^{N-q-1}\eta_{k+1},\,hf^{N-q-1}z_{k+1})\\&&+\rho_{\rm y}(hf^{N-q}\eta_{k},\,y_{k-q+1})\\
&\leq&\delta_{1}+q\delta_{1}\\ &=&(q+1)\delta_{1}
\end{eqnarray*}
which allows us to assert
\begin{eqnarray*}
\rho_{\rm y}(hf^{N-q}\eta_{k},\,y_{k-q+1})\leq
q\delta_{1}\quad\forall q\in\{1,\,2,\,\ldots,\,p+1\}\quad\forall
k\geq k_{p}+1\,.
\end{eqnarray*}
Since \eqref{eqn:claimNLproof1} holds with $p=1$, by induction we
can write
\begin{eqnarray}\label{eqn:claimNLproof2}
\rho_{\rm y}(hf^{N-q}\eta_{k},\,y_{k-q+1})\leq q\delta_{1}
\quad\forall q\in\{1,\,2,\,\ldots,\,N\}\quad\forall k\geq
k_{1}+N-1\,.
\end{eqnarray}
Now by \eqref{eqn:uniformobs} and \eqref{eqn:claimNLproof2} we can
write for $k\geq k_{1}+N-1$
\begin{eqnarray*}
\rho_{\rm x}(\xhat_{k+1},\,x_{k+1})
&=&\rho_{\rm x}(f^{N}\eta_{k},\,f^{N}\tx_{k})\\
&\leq&\alpha_{5}^{N}\rho_{\rm x}(\eta_{k},\,\tx_{k})\\
&\leq&\alpha_{5}^{N}\alpha_{3}^{-1}\sum_{i=0}^{N-1}\ell(hf^{i}\eta_{k},\,hf^{i}\tx_{k})\\
&=&\alpha_{5}^{N}\alpha_{3}^{-1}\sum_{q=1}^{N}\ell(hf^{N-q}\eta_{k},\,y_{k-q+1})\\
&\leq&\alpha_{5}^{N}\alpha_{3}^{-1}\sum_{q=1}^{N}\alpha_{2}\rho_{\rm y}(hf^{N-q}\eta_{k},\,y_{k-q+1})\\
&\leq&\alpha_{5}^{N}\alpha_{3}^{-1}\sum_{q=1}^{N}\alpha_{2}q\delta_{1}\\
&=&\alpha_{5}^{N}\alpha_{3}^{-1}\sum_{q=1}^{N}\alpha_{2}q\alpha_{1}^{-1}\delta\,.
\end{eqnarray*}
This proves our claim since we can define $\alpha_{8}\in\Kinf$ as
\begin{eqnarray*}
\alpha_{8}s:=\alpha_{5}^{N}\alpha_{3}^{-1}\sum_{q=1}^{N}\alpha_{2}q\alpha_{1}^{-1}s\,.
\end{eqnarray*}
Now we return to the proof of the theorem. By
\eqref{eqn:Jsumbounded} we can write for all $k\geq N-1$
\begin{eqnarray}\label{eqn:boundedness5}
\alpha_{4}J(\eta_{k},\,z_{k},\,y_{k})
&\leq&\sum_{i=0}^{N-2}\ell(hf^{i}z_{N-1},\,hf^{i}\tx_{N-1})\nonumber\\
&\leq&\sum_{i=0}^{N-2}\alpha_{2}\alpha_{6}\rho_{\rm x}(f^{i}z_{N-1},\,f^{i}\tx_{N-1})\nonumber\\
&\leq&\sum_{i=0}^{N-2}\alpha_{2}\alpha_{6}\alpha_{5}^{i}\rho_{\rm
x}(z_{N-1},\,x_{0})\nonumber\\
&\leq&\sum_{i=0}^{N-2}\alpha_{2}\alpha_{6}\alpha_{5}^{i}(\rho_{\rm
x}(z_{N-1},\,\xhat_{0})+\rho_{\rm
x}(\xhat_{0},\,x_{0}))\nonumber\\
&=&\sum_{i=0}^{N-2}\alpha_{2}\alpha_{6}\alpha_{5}^{i}(\rho_{\rm
x}(z_{N-1},\,f^{N-1}z_{0})+\rho_{\rm x}(\xhat_{0},\,x_{0}))\,.
\end{eqnarray}
Observe that for all $k\geq 1$ we can write
\begin{eqnarray}\label{eqn:boundedness6}
\rho_{\rm x}(z_{k},\,f^{k}z_{0})
&=&\rho_{\rm x}(f\eta_{k-1},\,f^{k}z_{0})\nonumber\\
&\leq&\alpha_{5}\rho_{\rm x}(\eta_{k-1},\,f^{k-1}z_{0})\nonumber\\
&\leq&\alpha_{5}(\rho_{\rm x}(\eta_{k-1},\,z_{k-1})+\rho_{\rm
x}(z_{k-1},\,f^{k-1}z_{0}))\,.
\end{eqnarray}
By \eqref{eqn:boundedness2} and \eqref{eqn:boundedness3} we can
write
\begin{eqnarray*}
\rho_{\rm x}(\eta,\,z)\leq\alpha_{9}\rho_{\rm x}(\xhat,\,x)
\end{eqnarray*}
where we define $\alpha_{9}\in\Kinf$ as
\begin{eqnarray*}
\alpha_{9}s:=\alpha_{3}^{-1}(\alpha_{2}\alpha_{6}s+\alpha_{2}(\alpha_{1}^{-1}\alpha_{2}\alpha_{6}s+\alpha_{6}s))\,.
\end{eqnarray*}
Hence we can proceed from \eqref{eqn:boundedness6} as
\begin{eqnarray}\label{eqn:boundedness7}
\rho_{\rm
x}(z_{k},\,f^{k}z_{0})&\leq&\alpha_{5}(\alpha_{9}\rho_{\rm
x}(\xhat_{k-1},\,x_{k-1})+\rho_{\rm
x}(z_{k-1},\,f^{k-1}z_{0}))\nonumber\\
&\leq&\alpha_{5}(\alpha_{9}\alpha_{7}^{k-1}\rho_{\rm
x}(\xhat_{0},\,x_{0})+\rho_{\rm x}(z_{k-1},\,f^{k-1}z_{0}))
\end{eqnarray}
where we used \eqref{eqn:boundedness4}. Then
\eqref{eqn:boundedness7} implies
\begin{eqnarray}\label{eqn:boundedness8}
\rho_{\rm x}(z_{N-1},\,f^{N-1}z_{0})&\leq&\gamma_{N-1}\rho_{\rm
x}(\xhat_{0},\,x_{0})
\end{eqnarray}
where we define $\gamma_{j}\in\Kinf$ recursively through
\begin{eqnarray*}
\gamma_{j+1}s:=\alpha_{5}(\alpha_{9}\alpha_{7}^{j}s+\gamma_{j}s)
\end{eqnarray*}
for $j\in\{1,\,2,\,\ldots\}$ with
$\gamma_{1}s:=\alpha_{5}\alpha_{9}s$. Now, by
\eqref{eqn:boundedness5} and \eqref{eqn:boundedness8} we can write
for all $k\geq N-1$
\begin{eqnarray}\label{eqn:boundedness9}
J(\eta_{k},\,z_{k},\,y_{k})\leq\alpha_{10}\rho_{\rm
x}(\xhat_{0},\,x_{0})
\end{eqnarray}
once we define $\alpha_{10}\in\Kinf$ as
\begin{eqnarray*}
\alpha_{10}s:=\alpha_{4}^{-1}\sum_{i=0}^{N-2}\alpha_{2}\alpha_{6}\alpha_{5}^{i}(\gamma_{N-1}s+s)\,.
\end{eqnarray*}
Note that \eqref{eqn:claimNL} and \eqref{eqn:boundedness9} allow
us to write
\begin{eqnarray*}
\rho_{\rm x}(\xhat_{k},\,x_{k})\leq\alpha_{8}\alpha_{10}\rho_{\rm
x}(\xhat_{0},\,x_{0})\quad\forall k\geq 2N-1\,.
\end{eqnarray*}
Moreover, by \eqref{eqn:boundedness4} we have
\begin{eqnarray*}
\rho_{\rm x}(\xhat_{k},\,x_{k})\leq\alpha_{7}^{k}\rho_{\rm
x}(\xhat_{0},\,x_{0})\quad\forall k\in\{0,\,1,\,\ldots,\,2N-2\}\,.
\end{eqnarray*}
Hence by defining $\alpha_{11}\in\Kinf$ as
\begin{eqnarray*}
\alpha_{11}s:=\max\left\{\alpha_{8}\alpha_{10}s,\,\max_{k\in\{0,\,\ldots,\,2N-2\}}\alpha_{7}^{k}s\right\}
\end{eqnarray*}
we can write
\begin{eqnarray*}
\rho_{\rm x}(\xhat_{k},\,x_{k})\leq\alpha_{11}\rho_{\rm
x}(\xhat_{0},\,x_{0})\quad\forall k\in\Natural
\end{eqnarray*}
which establishes the stability.

Now we prove convergence. From \eqref{eqn:Jsumbounded} we deduce
that $J(\eta_{k},\,z_{k},\,y_{k})\to 0$ as $k\to\infty$. Then
\eqref{eqn:claimNL} implies $\rho_{\rm x}(\xhat_{k},\,x_{k})\to 0$
as $k\to\infty$.
\end{proof}

\subsection{Tracker design}

Here we attempt to generalize the linear result of
Theorem~\ref{thm:asyconlin}, where an optimal tracker was
constructed through solving the following problem
\begin{eqnarray}\label{eqn:linearoptimalcost}
V(\xhat,\,x):=\min_{(w_{i})_{i=0}^{N-1}}\sum_{i=0}^{N-1}w_{i}^{T}R^{-1}w_{i}\quad\mbox{subject
to}\quad
\left\{\begin{array}{rcl}z_{0}&=&\xhat\\z_{i+1}&=&Az_{i}+Bw_{i}\\z_{N}&=&A^{N}x\end{array}\right.
\end{eqnarray}
where the control inputs $w_{0},\,w_{1},\,\ldots,\,w_{N-1}$ are
penalized via the quadratic stage cost $w\mapsto w^{T}R^{-1}w$.
Note that in a general nonlinear setting, imposing a direct
penalty on the control input may not be meaningful. For instance,
$w$ may just be an index that belongs to a finite set. For this
reason we will now try to express $V(\xhat,\,x)$ in a different
way that is more welcoming to generalization. If we assume that
$B$ is full column rank and let
$Q:=B(B^{T}B)^{-1}R^{-1}(B^{T}B)^{-1}B^{T}$, we can write
\begin{eqnarray*}
w_{i}^{T}R^{-1}w_{i}&=&w_{i}^{T}B^{T}QBw_{i}\\
&=&(z_{i+1}-Az_{i})^{T}Q(z_{i+1}-Az_{i})\,.
\end{eqnarray*}
Then we have
\begin{eqnarray*}
V(\xhat,\,x) &=& \min_{(z_{i})_{i=0}^{N}}\
\sum_{i=0}^{N-1}(z_{i+1}-Az_{i})^{T}Q(z_{i+1}-Az_{i})\\
&&\qquad\mbox{subject to}\quad
\left\{\begin{array}{rcl}z_{0}&=&\xhat\\z_{i+1}&\in&Az_{i}+{\rm
range}(B)\\z_{N}&=&A^{N}x\end{array}\right.
\end{eqnarray*}
which is the form we adopt for generalization. We remind the
reader that the optimal cost~\eqref{eqn:linearoptimalcost} enjoys
the following analytical expression
\begin{eqnarray*}
V(\xhat,\,x)=(\xhat-x)^{T}A^{NT}\left(\sum_{i=0}^{N-1}A^{i}BRB^{T}A^{iT}\right)^{-1}A^{N}(\xhat-x)\,.
\end{eqnarray*}
which is positive definite (for nonsingular $A$) with respect to
the error $e=\xhat-x$.

Consider now the system
\begin{eqnarray}\label{eqn:nonlinearcon}
x^{+}=fx
\end{eqnarray}
with $f:\X\to\X$ continuous and the tracker
\begin{eqnarray}\label{eqn:tracker}
\xhat^{+}=F(\xhat,\,u)
\end{eqnarray}
with $F:\X\times\U\to\X$. We assume $f\xhat\in F(\xhat,\,\U)$ for
all $\xhat\in\X$. Let $\ell:\X\times\X\to\Real_{\geq 0}$ and
$\alpha_{1},\,\alpha_{2}\in\Kinf$ satisfy $\alpha_{1}\rho_{\rm
x}(\xi,\,\zeta)\leq\ell(\xi,\,\zeta)\leq\alpha_{2}\rho_{\rm
x}(\xi,\,\zeta)$ for every $\xi,\,\zeta\in\X$. Since there is no
significant reason against it, we assume the symmetry
$\ell(\xi,\,\zeta)=\ell(\zeta,\,\xi)$. Let $N\geq 1$ and
$(\phi_{i})_{i=0}^{N}$ denote a sequence of functions
$\phi_{i}:\X\times\X\to\X$ satisfying
\begin{eqnarray}\label{eqn:phi}
(\phi_{i}(\xhat,\,x))_{i=0}^{N}&=&\arg\min_{(z_{i})_{i=0}^{N}}\
\sum_{i=0}^{N-1}\ell(z_{i+1},\,fz_{i})\nonumber\\
&&\qquad\mbox{subject to}\quad
\left\{\begin{array}{rcl}z_{0}&=&\xhat\\z_{i+1}&\in&F(z_{i},\,\U)\\z_{N}&=&f^{N}x\end{array}\right.
\end{eqnarray}
Note that $\phi_{0}(\xhat,\,x)=\xhat$ and
$\phi_{N}(\xhat,\,x)=f^{N}x$. Assuming $(\phi_{i})_{i=0}^{N}$
exists we can run the tracker~\eqref{eqn:tracker} so as to satisfy
$\xhat^{+}=\phi_{1}(\xhat,\,x)$ since by \eqref{eqn:phi} we have
$\phi_{1}(\xhat,\,x)\in F(\xhat,\,\U)$, i.e., for each pair
$(\xhat,\,x)$ we can find $u\in\U$ such that
$\phi_{1}(\xhat,\,x)=F(\xhat,\,u)$. Determining whether this
construction will actually work or not requires some analysis. As
usual in moving horizon feedback systems \cite{keerthi88} the
optimal cost $V:\X\times\X\to\Real_{\geq 0}$ defined below will be
of key importance in the analysis.

\begin{eqnarray}\label{eqn:V}
V(\xhat,\,x):=\min_{(z_{i})_{i=0}^{N}}\
\sum_{i=0}^{N-1}\ell(z_{i+1},\,fz_{i})\quad\mbox{subject to}\quad
\left\{\begin{array}{rcl}z_{0}&=&\xhat\\z_{i+1}&\in&F(z_{i},\,\U)\\z_{N}&=&f^{N}x\end{array}\right.
\end{eqnarray}

\begin{assumption}\label{assume:Vcontinuous}
The following hold.
\begin{enumerate}
\item The optimal cost~\eqref{eqn:V} is continuous and there exist
$\alpha_{3},\,\alpha_{4}\in\Kinf$ such that
\begin{eqnarray}\label{eqn:Vposdef}
\alpha_{3}\rho_{\rm x}(\xhat,\,x)\leq
V(\xhat,\,x)\leq\alpha_{4}\rho_{\rm x}(\xhat,\,x)
\end{eqnarray}
for all $\xhat,\,x\in\X$. \item The sequence of
functions~\eqref{eqn:phi} is unique and its second element
$\phi_{1}$ is continuous.
\end{enumerate}
\end{assumption}

\begin{remark}
Note that Assumption~\ref{assume:Vcontinuous} comes for free for
linear systems under the conditions of
Theorem~\ref{thm:asyconlin}, provided that the system matrix $A$
is nonsingular.
\end{remark}

\begin{theorem}\label{thm:asycon}
Consider the system~\eqref{eqn:nonlinearcon} and the
tracker~\eqref{eqn:tracker}. Let $N\geq 1$ and the control input
$u$ of the tracker satisfy $F(\xhat,\,u)=\phi_{1}(\xhat,\,x)$
where $\phi_{1}$ is the second term of the sequence of
functions~\eqref{eqn:phi}. If Assumption~\ref{assume:Vcontinuous}
holds, then for each $\varepsilon>0$ there exists $\delta>0$ such
that $\rho_{\rm x}(\xhat_{0},\,x_{0})\leq\delta$ implies
$\rho_{\rm x}(\xhat_{k},\,x_{k})\leq\varepsilon$ for all initial
conditions and $k\in\Natural$. Moreover, if the solutions
$\xhat_{k}$ and $x_{k}$ belong to a bounded region $\D\subset\X$
for all $k\in\Natural$, then $\rho_{\rm x}(\xhat_{k},\,x_{k})\to
0$ as $k\to\infty$.
\end{theorem}

\begin{proof}
The tracker dynamics being $\xhat^{+}=\phi_{1}(\xhat,\,x)$, we can
write by optimality
\begin{eqnarray}\label{eqn:decrease}
V(\xhat^{+},\,x^{+})-V(\xhat,\,x)\leq-\ell(\phi_{1}(\xhat,\,x),\,f\xhat)
\end{eqnarray}
whence it follows that $V(\xhat_{k},\,x_{k})\leq
V(\xhat_{0},\,x_{0})$ for all $k$. Hence the stability is
established by the assumed positive definiteness
\eqref{eqn:Vposdef} of the optimal cost $V$.

Now we show convergence under the assumption that both $\xhat_{k}$
and $x_{k}$ belong to a bounded region $\D$ for all $k$. By
\eqref{eqn:decrease} we can write $0\leq V(\xhat_{k+1},\,x_{k+1})
\leq V(\xhat_{k},\,x_{k})$ for all $k$. Therefore there exists
${\overline V}\geq 0$ such that $V(\xhat_{k},\,x_{k})\to{\overline
V}$ as $k\to\infty$. Note that establishing the convergence
$\rho_{\rm x}(\xhat_{k},\,x_{k})\to 0$ is equivalent to showing
that ${\overline V}=0$ thanks to \eqref{eqn:Vposdef}.

By $\ex\in\X^{2}$ let us denote the aggregate state $(\xhat,\,x)$.
Then we can write
$\ex^{+}=(\xhat^{+},\,x^{+})=(\phi_{1}(\xhat,\,x),\,fx)=:\ef\ex$.
Since both $f$ and $\phi_{1}$ are continuous, so is
$\ef:\X^{2}\to\X^{2}$. Also, the solution $\ex_{k}$ belongs to the
bounded region $\D^{2}$ for all $k$. Since the sequence
$(\ex_{k})_{k=0}^{\infty}$ is bounded it must have an accumulation
point $\ex^{*}=(\xhat^{*},\,x^{*})$. That is,
$(\ex_{k})_{k=0}^{\infty}$ must have a convergent subsequence
$(\ex_{k_{j}})_{j=0}^{\infty}$ satisfying $\ex_{k_{j}}\to\ex^{*}$
as $j\to\infty$ \cite{browder96}. Note that
$V\ex^{*}=V(\xhat^{*},\,x^{*})={\overline V}$ because $V$ is
continuous. Since $\ex^{+}=\ef\ex$ the sequence
$(\ef\ex_{k_{j}})_{j=0}^{\infty}$ must also be a subsequence of
$(\ex_{k})_{k=0}^{\infty}$. Moreover, $\ef\ex^{*}$ has to be an
accumulation point because $\ef$ is continuous. By induction
$\ef^{q}\ex^{*}$ is an accumulation point of
$(\ex_{k})_{k=0}^{\infty}$ for all $q\in\Natural$. Consequently
\begin{eqnarray}\label{eqn:johnnyguitar1}
V\ef^{q}\ex^{*}={\overline V}
\end{eqnarray}
for all $q\in\Natural$. By \eqref{eqn:decrease} and
\eqref{eqn:johnnyguitar1} we can write
\begin{eqnarray*}
\alpha_{1}\rho_{\rm
x}(\phi_{1}(\xhat^{*},\,x^{*}),\,f\xhat^{*})&\leq&\ell(\phi_{1}(\xhat^{*},\,x^{*}),\,f\xhat^{*})\\
&\leq&
V\ex^{*}-V\ef\ex^{*}\\&=&{\overline V}-{\overline V}\\
&=&0\,.
\end{eqnarray*}
Therefore $\phi_{1}(\xhat^{*},\,x^{*})=f\xhat^{*}$, which means
$\ef\ex^{*}=(f\xhat^{*},\,fx^{*})$. Employing induction we can
thus write
\begin{eqnarray}\label{eqn:johnnyguitar2}
\ef^{q}\ex^{*}=(f^{q}\xhat^{*},\,f^{q}x^{*})
\end{eqnarray}
i.e., $(f^{q}\xhat^{*},\,f^{q}x^{*})$ is an accumulation point.
Hence
\begin{eqnarray}\label{eqn:accumulation}
\phi_{1}(f^{q}\xhat^{*},\,f^{q}x^{*})=f^{q+1}\xhat^{*}
\end{eqnarray}
for all $q\in\Natural$.

Now, given a pair $(\xhat,\,x)$ let a sequence
$(z_{i})_{i=0}^{N}=(\xhat,\,z_{1},\,\ldots,\,z_{N-1},\,f^{N}x)$ be
said to be {\em feasible with respect to $(\xhat,\,x)$} if it
respects the constraints in \eqref{eqn:phi}. Also, we define
\begin{eqnarray*}
J(z_{i})_{i=0}^{N}:=\sum_{i=0}^{N-1}\ell(z_{i+1},\,fz_{i})\,.
\end{eqnarray*}
Then we can write by \eqref{eqn:accumulation}
\begin{eqnarray}\label{eqn:brandnew}
\lefteqn{V(f^{q}\xhat^{*},\,f^{q}x^{*})}\nonumber\\
&&=
J(\phi_{0}(f^{q}\xhat^{*},\,f^{q}x^{*}),\,\ldots,\,\phi_{N}(f^{q}\xhat^{*},\,f^{q}x^{*}))\nonumber\\
&&=
J(f^{q}\xhat^{*},\,f^{q+1}\xhat^{*},\,\phi_{2}(f^{q}\xhat^{*},\,f^{q}x^{*}),\,\ldots,\,\phi_{N-1}
(f^{q}\xhat^{*},\,f^{q}x^{*}),\,f^{N+q}x^{*})\nonumber\\
&&=
J(f^{q+1}\xhat^{*},\,\phi_{2}(f^{q}\xhat^{*},\,f^{q}x^{*}),\,\ldots,\,\phi_{N-1}
(f^{q}\xhat^{*},\,f^{q}x^{*}),\,f^{N+q}x^{*},\,f^{N+q+1}x^{*})\nonumber\\
&&=
J(\phi_{1}(f^{q}\xhat^{*},\,f^{q}x^{*}),\,\ldots,\,\phi_{N}(f^{q}\xhat^{*},\,f^{q}x^{*}),\,f^{N+q+1}x^{*})\,.
\end{eqnarray}
By \eqref{eqn:johnnyguitar1} and \eqref{eqn:johnnyguitar2} we have
$V(f^{q+1}\xhat^{*},\,f^{q+1}x^{*})=V(f^{q}\xhat^{*},\,f^{q}x^{*})$,
which allows us by \eqref{eqn:brandnew} to write
\begin{eqnarray*}
V(f^{q+1}\xhat^{*},\,f^{q+1}x^{*})
&=&J(\phi_{0}(f^{q+1}\xhat^{*},\,f^{q+1}x^{*}),\,\ldots,\,\phi_{N}(f^{q+1}\xhat^{*},\,f^{q+1}x^{*}))\\
&=&J(\phi_{1}(f^{q}\xhat^{*},\,f^{q}x^{*}),\,\ldots,\,\phi_{N}(f^{q}\xhat^{*},\,f^{q}x^{*}),\,f^{N+q+1}x^{*})\,.
\end{eqnarray*}
By \eqref{eqn:accumulation} the sequence
$(\phi_{1}(f^{q}\xhat^{*},\,f^{q}x^{*}),\,\ldots,\,\phi_{N}(f^{q}\xhat^{*},\,f^{q}x^{*}),\,f^{N+q+1}x^{*})$
is feasible with respect to $(f^{q+1}\xhat^{*},\,f^{q+1}x^{*})$.
Therefore the uniqueness condition stated in
Assumption~\ref{assume:Vcontinuous} implies the following equality
of sequences for all $q\in\Natural$
\begin{eqnarray*}
\lefteqn{(\phi_{1}(f^{q}\xhat^{*},\,f^{q}x^{*}),\,\ldots,\,\phi_{N}(f^{q}\xhat^{*},\,f^{q}x^{*}),\,f^{N+q+1}x^{*})}\\
&&=(\phi_{0}(f^{q+1}\xhat^{*},\,f^{q+1}x^{*}),\,\ldots,\,\phi_{N}(f^{q+1}\xhat^{*},\,f^{q+1}x^{*}))
\end{eqnarray*}
whence we can write
$\phi_{i}(f^{q}\xhat^{*},\,f^{q}x^{*})=f^{q+i}\xhat^{*}$ for
$i=1,\,\ldots,\,N$. That means we have
$\phi_{i+1}(f^{q}\xhat^{*},\,f^{q}x^{*})=f\phi_{i}(f^{q}\xhat^{*},\,f^{q}x^{*})$
for $i=0,\,\ldots,\,N-1$. Finally, by \eqref{eqn:johnnyguitar1}
and \eqref{eqn:johnnyguitar2}
\begin{eqnarray*}
{\overline V}&=&V(f^{q}\xhat^{*},\,f^{q}x^{*})\\
&=&\sum_{i=0}^{N-1}\ell(\phi_{i+1}(f^{q}\xhat^{*},\,f^{q}x^{*}),\,f\phi_{i}(f^{q}\xhat^{*},\,f^{q}x^{*}))\\
&=&\sum_{i=0}^{N-1}\ell(f\phi_{i}(f^{q}\xhat^{*},\,f^{q}x^{*}),\,f\phi_{i}(f^{q}\xhat^{*},\,f^{q}x^{*}))\\
&\leq&\sum_{i=0}^{N-1}\alpha_{2}\rho_{\rm x}(f\phi_{i}(f^{q}\xhat^{*},\,f^{q}x^{*}),\,f\phi_{i}(f^{q}\xhat^{*},\,f^{q}x^{*}))\\
&=&0
\end{eqnarray*}
which was to be shown.
\end{proof}

\begin{remark}
For the special (yet important) case where the trajectory of the
system~\eqref{eqn:nonlinearcon} to be tracked is constant, i.e.,
$x_{k}=x_{\rm eq}$ for all $k$, the boundedness condition required
in Theorem~\ref{thm:asycon} to establish convergence $\rho_{\rm
x}(\xhat_{k},\,x_{k})\to 0$ need not be explicitly assumed for it
is implied by \eqref{eqn:Vposdef} and \eqref{eqn:decrease}. In
other words, to establish the regulation of an equilibrium point
$x_{\rm eq}=fx_{\rm eq}$ Assumption~\ref{assume:Vcontinuous} is
sufficient.
\end{remark}

\bibliographystyle{plain} \bibliography{references}
\end{document}